\documentclass[12pt]{amsart}

\usepackage{graphicx}
\usepackage{amsthm}
\usepackage{amsmath, amssymb}
\swapnumbers

\theoremstyle{plain}
\newtheorem{Thm}{Theorem}

\newtheorem{Lem}[Thm]{Lemma}

\theoremstyle{definition}
\newtheorem{Def}[Thm]{Definition}

\begin{document}

\title[MCGs of once-stabilized Heegaard splittings]{Mapping class groups of once-stabilized Heegaard splittings}

\author{Jesse Johnson}
\address{\hskip-\parindent
        Department of Mathematics \\
        Oklahoma State University \\
        Stillwater, OK 74078
        USA}
\email{jjohnson@math.okstate.edu}

\subjclass{Primary 57N10}
\keywords{Heegaard splitting, mapping class group, curve complex}

\thanks{This project was supported by NSF Grant DMS-1006369}

\begin{abstract}
We show that if a Heegaard splitting is the result of stabilizing a high distance Heegaard splitting exactly once then its mapping class group is finitely generated.
\end{abstract}

\maketitle

A \textit{Heegaard splitting} for a compact, connected, closed, orientable 3-manifold $M$ is a triple $(\Sigma, H^-, H^+)$ where $\Sigma$ is a compact, separating surface in $M$ and $H^-$, $H^+$ are handlebodies in $M$ such that $M = H^- \cup H^+$ and $\partial H^- = \Sigma = H^- \cap H^+ = \partial H^+$.  The \textit{mapping class group} $Mod(M,\Sigma)$ of the Heegaard splitting is the group of homeomorphisms $f : M \rightarrow M$ that take $\Sigma$ onto itself, modulo isotopies that fix $\Sigma$ setwise.  

Heegaard splitings of distance greater than 3 are known to have finite mapping class groups \cite{meddist} and certain distance two Heegaard splittings are known to have virtually cyclic mapping class groups \cite{me:books}.  By distance, we mean the distance $d(\Sigma)$ defined by Hempel~\cite{hempel}, which we will review below.  However, for stabilized (distance zero) Heegaard splittings, the problem of understanding their mapping class groups is much harder.  Genus-two Heegaard splittings of the 3-sphere have finitely presented mapping class groups \cite{akbas:gen2}\cite{cho:gen2}\cite{schar:gen2}, but determining whether the mapping class groups of higher genus splittings of $S^3$ are finitely generated has proved to be a very difficult problem.  We will study Heegaard splittings that results from stabilizing a high distance Heegaard splitting exactly once:

\begin{Thm}
\label{mainthm}
Let $(\Sigma', H^-_{\Sigma'}, H^+_{\Sigma'})$ be a Heegaard surface with genus $g$ and distance $d(\Sigma') > 2g+2$.  If $(\Sigma, H^-_\Sigma, H^+_\Sigma)$ is the Heegaard splitting that results from stabilizing $\Sigma'$ exactly once then $Mod(M, \Sigma)$ is finitely generated.  
\end{Thm}

In fact, we will describe an explicit generating set in Section~\ref{handlebodysect}. Because every automorphism of $(M, \Sigma)$ is an automorphism of $M$, there is a canonical map $i : Mod(M, \Sigma) \rightarrow Mod(M)$.  We will write $Isot(M, \Sigma)$ for the kernel of the map $i$, and will call this the \textit{isotopy subgroup}.  When $M$ is hyperbolic, $Isot(M, \Sigma)$ will be a finite index subgroup of $Mod(M, \Sigma)$.

To prove Theorem~\ref{mainthm}, we show that $Isot(M, \Sigma)$ is finitely generated in a very specific way.  This will imply that $Mod(M, \Sigma)$ is finitely generated because $d(\Sigma') > 2g+2 > 3$, so $M$ must be hyperbolic, and $Isot(M, \Sigma)$ is finite index in $Mod(M, \Sigma)$. We will see that $Isot(M, \Sigma)$ is generated by two subgroups that have been recently classified by Scharlemann~\cite{schar:hbdy}, with a well understood intersection.  We conjecture that $Isot(M, \Sigma)$ is in fact a free product with amalgamation, but we are unable to prove this.

Note that every element in $Isot(M, \Sigma)$ is determined by an isotopy of $\Sigma$ in $M$, i.e.\ a continuous family of embedded surfaces $\{\Sigma_r\}$ such that $\Sigma_0 = \Sigma_1 = \Sigma$.  The two subgroups that generate $Isot(M, \Sigma)$ will be defined as follows:

The stabilized Heegaard surface $\Sigma$ can be isotoped into either of the handlebodies $H^\pm_{\Sigma'}$ bounded by the original Heegaard surface $\Sigma'$. After such an isotopy it forms a Heegaard surface for the handlebody.  Define $St^-(\Sigma)$ as the subgroup of $Isot(M, \Sigma)$ corresponding to isotopies entirely in $H^-_{\Sigma'}$ and let $St^+(\Sigma)$ be the subgroup corresponding to isotopies in $H^+_{\Sigma'}$.

These two subgroups are the isotopy subgroups of the mapping class groups for $\Sigma$, thought of as a Heegaard splitting for $H^-_{\Sigma'}$ or $H^+_{\Sigma'}$.  Scharlemann~\cite{schar:hbdy} shows that such groups are finitely generated.  We will show that $Isot(M, \Sigma)$ is generated by these two subgroups.

We review this generating set in Section~\ref{handlebodysect}, then determine a condition on isotopies of $\Sigma$ in $M$ that will guarantee that the corresponding element of $Isot(M, \Sigma)$ is generated by elements of $St^-(\Sigma) \cup St^+(\Sigma)$.  The proof is based on the double sweep-out machinery developed by Cerf~\cite{cerf}, Rubinstein-Scharlemann~\cite{rub:compar} and the author~\cite{me:stabs}, which we review Sections~\ref{sweepsect} and~\ref{facingsect}.  The proof of this Lemma and Theorem~\ref{mainthm} are completed in Section~\ref{proofsect}.

\section{Mapping class groups in handlebodies}
\label{handlebodysect}

A genus $g$ handlebody $H$ has, up to isotopy, a unique Heegaard splitting for each genus $h \geq g$~\cite{schthom:crossi}.  For $h = g$, this Heegaard splitting is a boundary parallel surface, which cuts $H$ into a genus $g$ handle-body and a trivial compression body $\partial H \times [0,1]$.  The higher genus Heegaard splittings come from adding unknotted handles to the genus $g$ Heegaard splitting, as in Figure~\ref{fig:stabilizedsplitting}. This construction is called \textit{stabilization}.
\begin{figure}[htb]
  \begin{center}
  \includegraphics[width=3in]{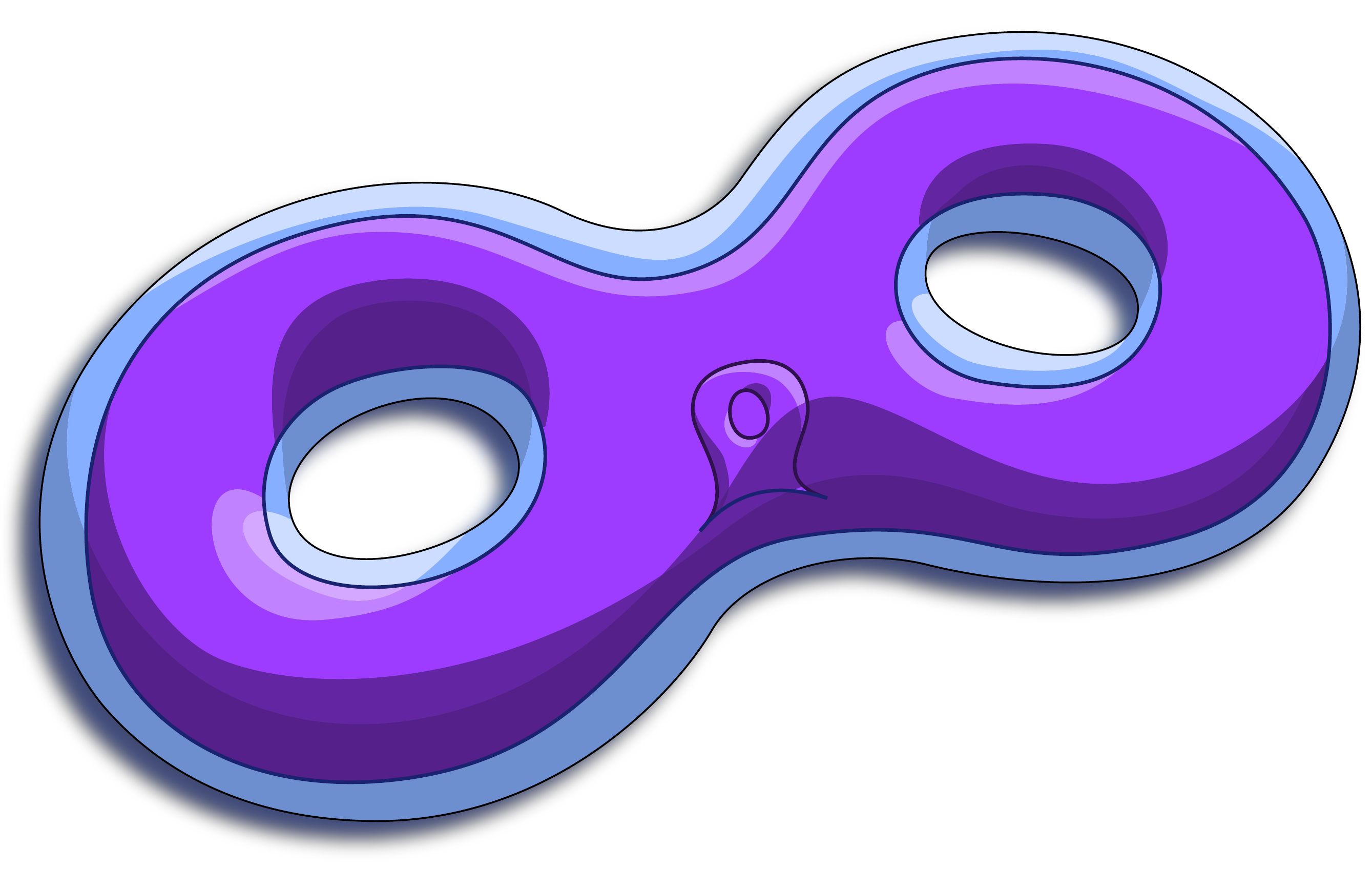}
  \caption{The genus $3$ Heegaard surface for a genus two handlebody.}
  \label{fig:stabilizedsplitting}
  \end{center}
\end{figure}

An equivalent way to construct a genus $g + k$ Heegaard splitting of a handlebody $H$ is to take $k$ boundary parallel, properly embedded arcs $\alpha_1,\dots,\alpha_k$ in $H$ and let $\Sigma$ be the boundary of a regular neighborhood of $\partial H \cup \alpha_1 \cup \dots \cup \alpha_k$.

Let $\Sigma$ be the genus $g + 1$ Heegaard splitting for $H$. Scharlemann~\cite{schar:hbdy} has determined a very simple generating set for the isotopy subgroup $Isot(H, \Sigma)$ in terms of the boundary parallel arc $\alpha \subset H$ that defines $\Sigma$.  The surface $\Sigma$ bounds a genus $g+1$ handlebody on one side and a compression body with a single one-handle on the other side.  There is exactly one non-separating compressing disk in the compression body, dual to the arc $\alpha$, so the isotopy subgroup of $\Sigma$ can be described entirely in terms of isotopies of $\alpha$.

Let $D \subset H$ be a disk whose boundary consists of the arc $\alpha$ and an arc in $\partial H$ (since $\alpha$ is unknotted) and let $E_1,\dots,E_g$ be a collection of compressing disks for $\partial H$ that are disjoint from $D$ and cut $H$ into a single ball.  By Theorem 1.1 in~\cite{schar:hbdy}, $Isot(H, \Sigma)$ is generated by the following two subgroups. (We use slightly different notation here.) \\

\begin{enumerate}
  \item Let $\mathfrak{F}(H, \Sigma)$ be the subgroup generated by isotoping the disk $D$ around $H$, and dragging $\alpha$ with it.  Because $D \cap \partial H$ is an arc, each element is defined by a path in $\partial H$, modulo spinning around a regular neighborhood of $D \cap \partial H$.  Thus this subgroup is isomorphic to an extension of $\pi_1(\partial H)$ by the integers. \\
  \item Let $\mathfrak{A}(H, \Sigma)$ be the subgroup generated by fixing one endpoint of $\alpha$ and dragging the other around in the complement of a collection of properly embedded disks $D_1,\dots,D_g$ whose complement in $H$ is a single ball.  Because the complement of the disks is the ball, any path of one endpoint can be extended to an isotopy of $\alpha$ that ends back where it started.  Any such mapping class is determined by an element of the fundamental group of the planar surface $\partial H \setminus (\bigcup D_i)$, so this group is isomorphic to the fundamental group of the planar surface. \\
\end{enumerate}
\begin{figure}[htb]
  \begin{center}
  \includegraphics[width=4.5in]{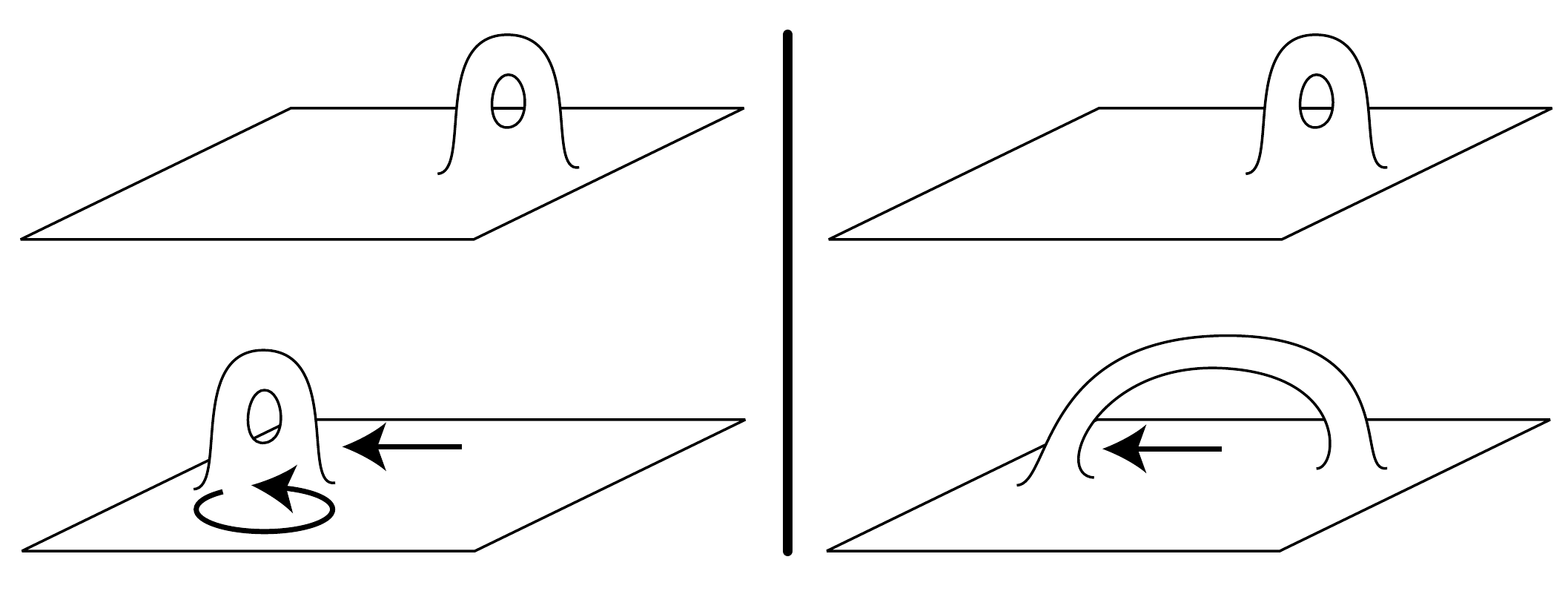}
  \caption{The isotopies that generate $\mathfrak{F}(H, \Sigma)$ and $\mathfrak{A}(H, \Sigma)$.}
  \label{fig:scharlemoves}
  \end{center}
\end{figure}

In particular, each of these subgroups is finitely generated, so $Isot(H, \Sigma)$ is finitely generated.  Each of the subgroups $St^\pm(\Sigma) \subset Isot(M, \Sigma)$ is isomorphic to $Isot(H, \Sigma)$.  Their intersection contains the subgroups $\mathfrak{F}(H^\pm_{\Sigma'}, \Sigma)$ in each handlebody because the isotopies defining these subgroups can be carried out within a regular neighborhood of the boundary.

We would like to show that if $\Sigma'$ is a high distance Heegaard splitting then every element of $Isot(M, \Sigma)$ is a product of elements of $St^-(\Sigma)$ and $St^+(\Sigma)$.  Fix spines $K^-$, $K^+$ of the handebodies $H^-_{\Sigma'}$, $H^+_{\Sigma'}$. The key will be the following Lemma:

\begin{Lem}
\label{ifneverbothlem}
Let $\{\Sigma_r |\ r \in [0,1]\}$ be an isotopy of the surface $\Sigma$ and assume there is a sequence of values $0 = r_0 < r_1 < \dots < r_k = 1$ with the property that for $r \in [r_i, r_{i+1}]$, the surface $\Sigma_r$ is disjoint from $K^-$ when $i$ is even and disjoint from $K^+$ when $i$ is odd.  Moreover, assume that each $\Sigma_{r_i}$ is a Heegaard surface for the complement of the two spines.  Then the element of $Isot(M, \Sigma)$ defined by $\{\Sigma_r\}$ is generated by the elements of $St^-(\Sigma) \cup St^+(\Sigma)$.
\end{Lem}

\begin{proof}
If every $\Sigma_r$ in the isotopy is disjoint from both spines $K^\pm$ then there is a value $\epsilon$ such that each $\Sigma_r$ is contained in $f^{-1}([\epsilon, 1-\epsilon])$.  This isotopy is conjugate to an isotopy in $H^-_{\Sigma'}$ as well as to an isotopy in $H^+_{\Sigma'}$, so it determines an element in the subgroup $St^-(\Sigma) \cap St^+(\Sigma)$.

If an isotopy $\{\Sigma_r\}$ ends with $\Sigma_1$ disjoint from both spines and isotopic to but not equal to $\Sigma$, then we can extend the isotopy so that $\Sigma_{1+\delta} = \Sigma$.  This extension is not unique, but it is well defined up to multiplication by elements in $St^-(\Sigma) \cap St^+(\Sigma)$.  Thus such an isotopy ending disjoint from $K^- \cup K^+$ determines a coset of the intersection.

If an isotopy is disjoint from $K^+$ and ends with the image of $\Sigma$ isotopic to a Heegaard surface for the complement of $K^+$ and $K^-$ then it can be extended to an isotopy that takes $\Sigma$ onto itself. Moreover, because the isotopy is disjoint from $K^+$, it is disjoint from the closure of some regular neighborhood $N$ of $K^+$ and $N$ is isotopic to $H^+(\Sigma')$. Thus by conjugating the isotopy of $\Sigma$ with an ambient isotopy that takes $N$ onto $H^+(\Sigma')$, we can turn the original isotopy into an isotopy of $\Sigma$ in $H^-_{\Sigma'}$. Such an isotopy determines a coset of the intersection subgroup inside $St^-(\Sigma)$.  Similarly, if the isotopy is disjoint from $K^-$, it determines a coset inside $St^+(\Sigma)$.  

We have assumed that our isotopy $\{\Sigma_r\}$ can be cut into finitely many sub-intervals such that in each interval, $\Sigma_r$ is always disjoint from $K^-$ or always disjoint from $K^+$. The restriction of the isotopy $\{\Sigma_r\}$ to each interval determines a coset of $St^-(\Sigma) \cap St^+(\Sigma)$ in either $St^-(\Sigma)$ or $St^+(\Sigma)$.  The element of $Isot(M, \Sigma)$ defined by the entire isotopy is a product of representatives for these cosets, and is thus in the subgroup generated by $St^-(\Sigma) \cup St^+(\Sigma)$.
\end{proof}

To prove Theorem~\ref{mainthm}, we must show that we can always find an isotopy of this form.  The rest of the paper will be devoted to proving this:

\begin{Lem}
\label{neverbothlem}
If $d(\Sigma')$ is greater than $2g+2$ then every element of $Isot(M, \Sigma)$ is represented by an isotopy satisfying the conditions of Lemma~\ref{ifneverbothlem}.
\end{Lem}

\section{Sweep-outs and graphics}
\label{sweepsect}

A \textit{sweep-out} is a smooth function $f : M \rightarrow [-1,1]$ such that for each $t \in (-1,1)$, the level set $f^{-1}(t)$ is a closed surface.  Moreover, $f^{-1}(-1)$ and $f^{-1}(1)$ is each a graph, called a \textit{spine} of the sweep-out.  The preimages $f^{-1}([-1,t])$ and $f^{-1}([t,1])$ are handlebodies for each $t \in (-1,1)$ so each level surface $f^{-1}(t)$ is a Heegaard surface for $M$ and the spines of the sweep-outs are spines of the two handlebodies in this Heegaard splitting.   See~\cite{me:stabs} for a more detailed description of the methods described in this section.

We will say that a sweep-out \textit{represents} a Heegaard splitting $(\Sigma', H^-_{\Sigma'}, H^+_{\Sigma'})$ if $f^{-1}(-1)$ is isotopic to a spine for $H^-_{\Sigma'}$ and $f^{-1}(1)$ is isotopic to a spine for $H^+_{\Sigma'}$. The level surfaces $f^{-1}(t)$ of such a sweep-out will be isotopic to $\Sigma'$.  Because the complement of the spines of a Heegaard splitting is a surface cross an interval, we can construct a sweep-out for any Heegaard splitting, i.e. we have the following:

\begin{Lem}
Every Heegaard splitting of a compact, connected, closed orientable, smooth 3-manifold is represented by a sweep-out.
\end{Lem}

Given two sweep-outs, $f$ and $h$, their product is a smooth function $f \times h : M \rightarrow [-1,1] \times [-1,1]$.  (That is, we define $(f \times h)(x) = (f(x),h(x))$.)  The \textit{discriminant set} for $f \times h$ is the set of points where the level sets of the two functions are tangent.

Generically, the discriminant set will be a one dimensional smooth submanifold in the complement in $M$ of the spines~\cite{Kob:disc,mather}.  The function $f \times h$ defines a piecewise smooth map from this collection of arcs and loops into the square $[-1,1] \times [-1,1]$. At the finitely many non-smooth points, the image is a cusp, which we will think of as a valence-two vertex. At the finitely many points where the restriction of $f$ is two-to-one, we see a crossing which we will think of as a valence-four vertex. There are also valence-one and -two vertices in the boundary of the square. The resulting graph is called the \textit{Rubinstein-Scharlemann graphic} (or just the \textit{graphic} for short). Kobayashi-Saeki's approach~\cite{Kob:disc} uses singularity theory to recover the machinery originally constructed by Rubinstein and Scharlemann in ~\cite{rub:compar} using Cerf theory~\cite{cerf:strat}.

\begin{Def}
The function $f \times h$ is \textit{generic} if the discriminant set is a smooth one-dimensional manifold and each arc $\{t\} \times [-1,1]$ or $[-1,1] \times \{s\}$ contains at most one vertex of the graphic.
\end{Def}

Kobayashi and Saeki~\cite{Kob:disc} have shown that after an isotopy of $f$ and $h$, we can assume that $f \times h$ is generic.  The author has generalized this to an isotopy of sweep-outs as follows:

\begin{Lem}[Lemma 34 in~\cite{me:stabs}]
\label{isotopesweepslem}
Every isotopy of the sweep-out $h$ is conjugate to an isotopy $\{h_r\}$ so that the graphic defined by $f$ and $h_r$ is generic for all but finitely many values of $r$.  At the finitely many non-generic points, one of six changes can occur to the graphic, indicated in Figure~\ref{fig:graphicmoves}:
\begin{enumerate}
\item A pair of cusps forming a bigon may cancel with each other of be created.
\item A pair of cusps adjacent to a common crossing may cancel or be created (similar to a Reidemeister one-move).
\item Three crossing may perform a Reidemeister three-move.
\item Two parallel edges may pinch together to form a pair of cusps, or vice versa.
\item Two edges may perform a Reidemeister two-move.
\item A cusp may pass across an edge.
\end{enumerate}
\end{Lem}
\begin{figure}[htb]
  \begin{center}
  \includegraphics[width=3.5in]{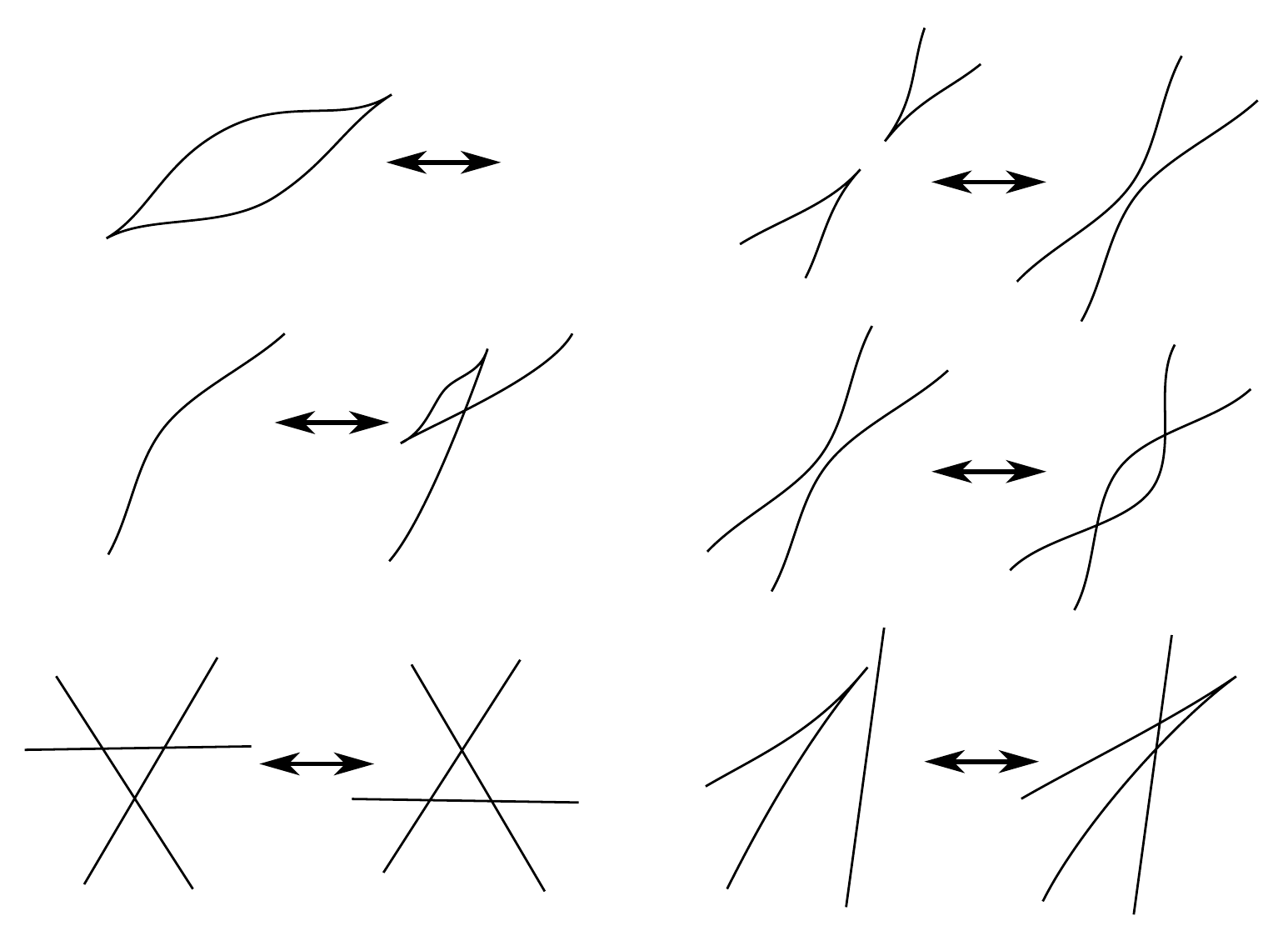}
  \caption{Changes to the graphic under an isotopy of $f$.}
  \label{fig:graphicmoves}
  \end{center}
\end{figure}

For a more details description of these three moves, see~\cite{me:stabs}.

\section{Spanning and Splitting}
\label{facingsect}

Let $f$ and $h$ be sweep-outs. For each $s \in (-1,1)$, define $\Sigma_s = h^{-1}(s)$, $H^-_s = h^{-1}([-1,s])$ and $H^+_s = h^{-1}([s,1])$.  Similarly, for $t \in (-1,1)$, define $\Sigma'_t = f^{-1}(t)$.  Following~\cite{me:stabs}, we will say that $\Sigma'_t$ is \textit{mostly above} $\Sigma_s$ if each component of $\Sigma'_t \cap H^-_s$ is contained in a disk subset of $\Sigma'_t$.  Similarly, $\Sigma'_t$ is \textit{mostly below} $\Sigma_s$ if each component of $\Sigma'_t \cap H^+_s$ is contained in a disk in $\Sigma'_t$.  

Figure~\ref{fig:abovebelow} shows the three possible positions for $\Sigma'_t$ (shown in blue) with respect to $\Sigma_s$ for three different values of $s$ (outlined in red). For the highest value of $s$, $\Sigma'_t$ is mostly below $\Sigma_s$. For the lowest value, it's mostly above, and for the middle value (in which $\Sigma_s$ looks like a quadrilateral, $\Sigma'_t$ is neither mostly above nor mostly below.
\begin{figure}[htb]
  \begin{center}
  \includegraphics[width=3.5in]{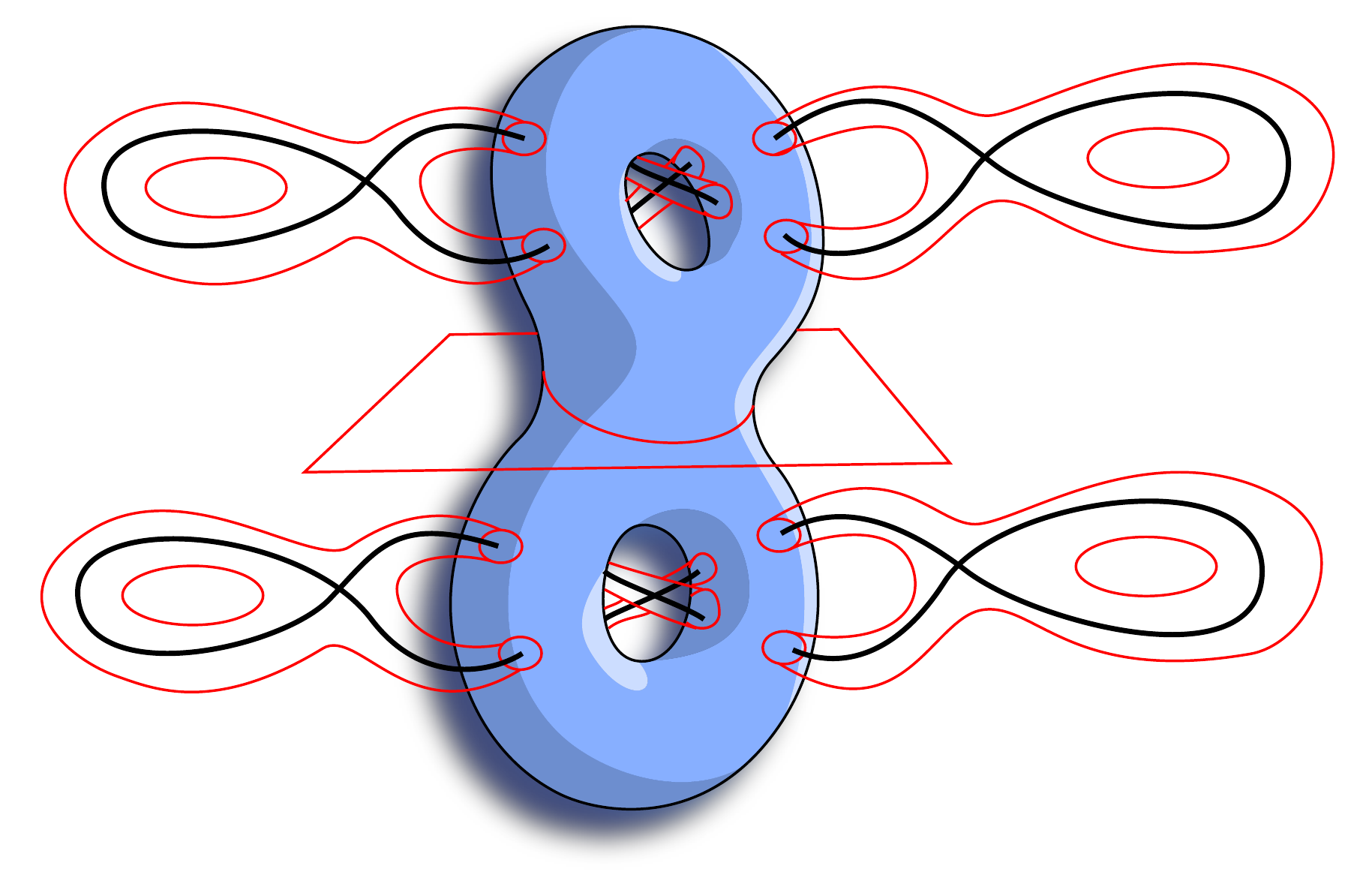}
  \caption{Changes to the graphic under an isotopy of $f$.}
  \label{fig:abovebelow}
  \end{center}
\end{figure}

Let $R_a \subset (-1,1) \times (-1,1)$ be the set of all values $(t,s)$ such that $\Sigma'_t$ is mostly above $\Sigma_s$.  Let $R_b \subset (-1,1) \times (-1,1)$ be the set of all values $(t,s)$ such that $\Sigma'_t$ is mostly below $\Sigma_s$.  For any fixed $t \in (-1,1)$, there will be values $a,b$ such that $\Sigma'_t$ will be mostly above $\Sigma_s$ if and only if $s \in [-1,a)$ and mostly above $\Sigma_s$ if and only if $s \in (b,1]$. In particular, both regions will be vertically convex.

As noted in~\cite{me:stabs}, the closure of $R_a$ in $(-1,1) \times (-1,1)$ is bounded by arcs of the Rubinstein-Scharlemann graphic, as is the closure of $R_b$.  The closures of $R_a$ and $R_b$ are disjoint (as long as the level surfaces of $f$ have genus at least two.)

\begin{Def}
\label{ffdef}
Given a generic pair $f$, $h$, we will say $h$ \textit{spans} $f$ if there is a horizontal arc $[-1,1] \times \{s\}$ that intersects the interiors of both regions $R_a$ and $R_b$, as on the left side of Figure~\ref{fig:spansplit}.  Otherwise, we will say that $h$ \textit{splits} $f$, as on the right of the figure.
\end{Def}
\begin{figure}[htb]
  \begin{center}
  \includegraphics[width=2.5in]{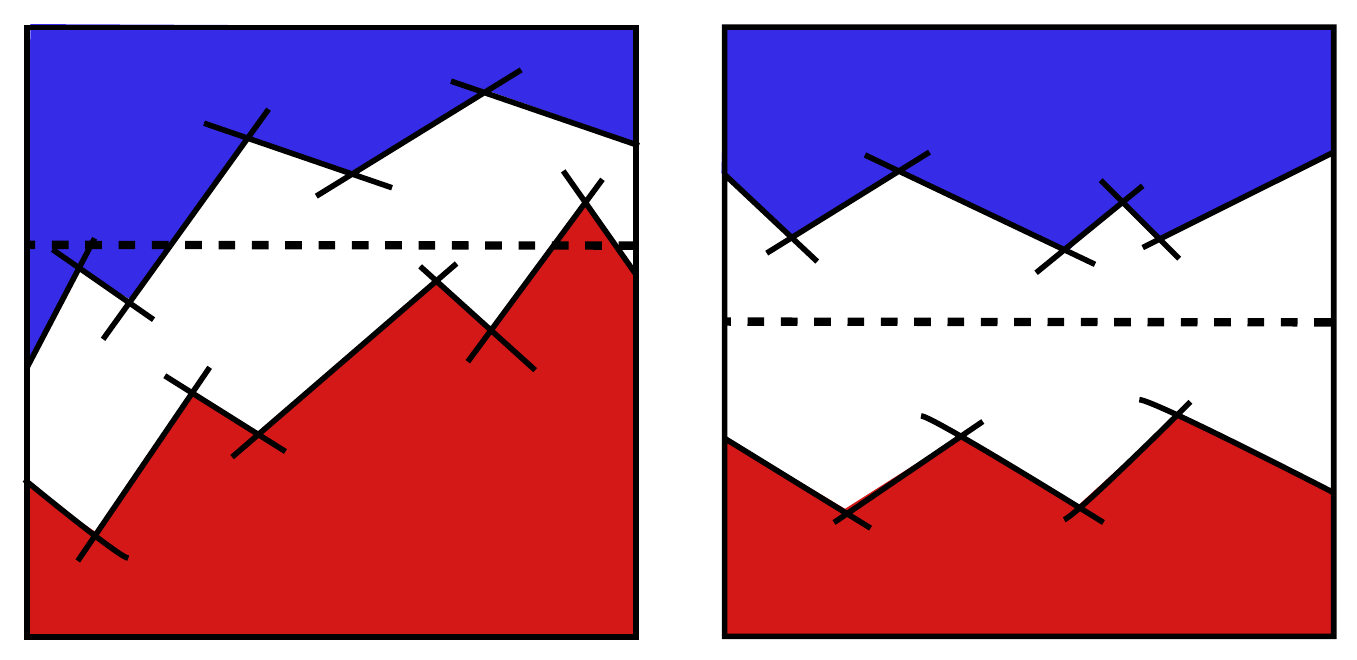}
  \caption{Spanning and splitting forms of the graphic.}
  \label{fig:spansplit}
  \end{center}
\end{figure}

Let $f$ be a sweep-out representing $(\Sigma', H^-_{\Sigma'}, H^+_{\Sigma'})$.  By Lemma~16 in~\cite{me:stabs} we can choose a sweep-out $h$ for $(\Sigma, H^-_\Sigma, H^+_\Sigma)$ that spans $f$.

We can identify $\Sigma$ with any level surface of $h$ by an isotopy and we will choose a specific surface in a moment.  For now, note that once we have chosen a level surface of $h$ to represent $\Sigma$, every element of $Isot(M, \Sigma)$ is represented by an isotopy of $\Sigma$, which can be extended to an ambient isotopy of $h$.  By Lemma~\ref{isotopesweepslem}, we can choose the induced family of sweep-outs $\{h_r\}$ to be generic for all but finitely many values of $r$.  At all these values, $h_r$ either spans or splits $f$. We will rule out splitting using the following Lemma:

\begin{Lem}[Lemma 27 in~\cite{me:stabs}]
\label{nospanninglem}
If $h$ splits $f$ then $d(\Sigma')$ is at most twice the genus of $\Sigma$.
\end{Lem}

Note that the notation here is slightly different than that used in~\cite{me:stabs}. In particular, the surfaces $\Sigma$ and $\Sigma'$ play the opposite roles in~\cite{me:stabs} as they do here.  In the statement of Theorem~\ref{mainthm}, we assume that $d(\Sigma')$ is strictly greater than $2g+2$, so Lemma~\ref{nospanninglem} implies that $h_r$ spans $f$ for every generic value of $r$.  This will be the key to proving Lemma~\ref{neverbothlem}

\section{Bicompressible surfaces in handlebodies}
\label{bicompsect}

Recall that a two-sided surface $S \subset M$ is \textit{bicompressible} if there are compressing disks for $S$ on both sides of the surface.  We will say that $S$ is \textit{reducible} if there is a sphere $P \subset M$ such that $P \cap S$ is a single loop that is essential in $S$.  This term is usually used only for Heegaard surfaces, but we will apply it here to general bicompressible surfaces.

\begin{Lem}
\label{bicomplem}
Let $F$ be a closed, orientable, genus $g$ surface and $S \subset (F \times [0,1])$ a closed, embedded, bicompressible surface of genus $g+1$ that separates $F \times \{0\}$ from $F \times \{1\}$.  Then $S$ is reducible.
\end{Lem}

\begin{proof}
Let $C^-$ be the closure of the component of $F \times [0,1] \setminus S$ adjacent to $F \times \{0\}$ and let $C^+$ be closure of the other component.  Because $S$ is bicompressible, there are compressing disks $D^- \subset C^-$ and $D^+ \subset C^+$ for $S$.  

If $\partial D^-$ is non-separating in $S$ then compressing $S$ across $D^-$ produces a genus $g$ surface $S^-$ that separates $F \times \{0\}$ from $F \times \{1\}$.  Because $F$ is also genus $g$, the surface $S^-$ must be isotopic to $F \times \{\frac{1}{2}\}$.  In other words, $S^-$ separates $F \times [0,1]$ into two pieces homeomorphic to $F \times [0,1]$.  If we reattach the tube to produce $S$ from $S^-$, we see that $C^-$ is a compression body that results from attaching a single one-handle to $F \times [0,1]$.  

The same argument applies to $D^+$ and $C^+$.  Thus if we can choose $D^-$ and $D^+$ to be non-separating, $S$ will be a genus $g+1$ Heegaard surface for $F \times [0,1]$.  Every genus $g+1$ Heegaard surface for $F \times [0,1]$ is reducible by Scharlemann-Thompson's classification of Heegaard splittings for surface-cross-intervals~\cite{stcross}, so in this case we conclude that $S$ is reducible.  

Otherwise, assume without loss of generality that every compressing disk for $S$ in $C^-$ is separating.  If we compress $S$ along such a disk $D^-$ then the resulting surface $S^-$ consists of a genus $g$ component and a torus.  As above, the genus $g$ component must be isotopic to $F \times \{\frac{1}{2}\}$ so that the torus component of $S^-$ is contained in $F \times [\frac{1}{2}, 1]$.

Because $F \times [\frac{1}{2}, 1]$ is atoroidal, the torus component $T$ $S^-$ can be compressed to form a sphere $T'$. Because $F \times [\frac{1}{2}, 1]$ is irreducible, $T'$ bounds a ball and we can recover $T$ by attaching a tube to $T'$. Thus either bounds a solid torus or is contained in a ball (and bounds a knot complement), depending on which side of $T'$ the tube us attached.  In either case, $T$ has a (non-separating) compressing disk $D^+$.  There is an arc $\alpha$ dual to the disk $D^-$ from the genus $g$ component of $S^-$ to $T$.  Because $D^-$ and $D^+$ are on opposite sides of $S$, the arc most be adjacent to $T$ on the same side of the surface as the disk $D^+$.  This is impossible if $T$ bounds a solid torus on the side containing $D^+$.  Thus $T$ must be contained in a ball that intersects the arc $\alpha$ in a single point.  This sphere will intersect $S$ in a single essential loop, so $S$ is reducible (though not necessarily a Heegaard surface).
\end{proof}

\section{The Proof of Theorem~\ref{mainthm}}
\label{proofsect}

\begin{proof}[Proof of Lemma~\ref{neverbothlem}]
By assumption, $d(\Sigma') > 2g + 2$ where $g$ is the genus of $\Sigma'$. Because $\Sigma$ is a stabilization of $\Sigma'$, its genus is $g+1$, so by Lemma~\ref{nospanninglem}, the graphic $f \times h_r$ can never be spanning. Thus there is a value $s_r$ for each sweep-out $h_r$ such that the horizontal line $[0,1] \times \{s_r\}$ passes through both regions $R_a$, $R_b$ of the graphic.  Moreover, we can choose the values $s_r$ so that they vary continuously with $r$ and $s_0 = s_1$.  We will further choose values $a_r$, $b_r$ that vary continuously, except for finitely many jumps, such that $(a_r, s_r) \in R_a$ and $(b_r, s_r) \in R_b$. (A jump occurs when $a_r$ or $b_r$ is in a ``tooth'' of $R_a$ or $R_b$ that is moved away from the horizontal arc, as in Figure~\ref{fig:spanjump}, and we choose a new point in a different tooth that still intersects the arc.)
\begin{figure}[htb]
  \begin{center}
  \includegraphics[width=2.5in]{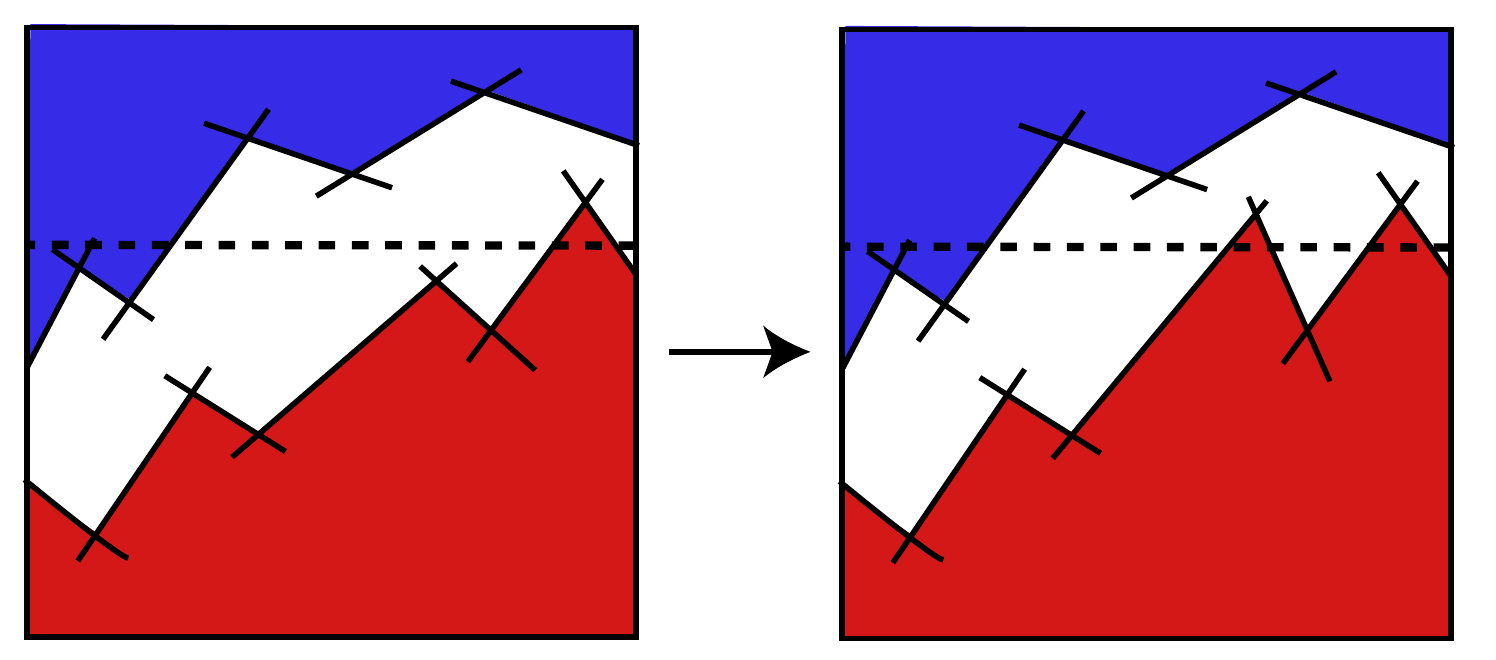}
  \caption{A jump in the value of $b_r$.}
  \label{fig:spanjump}
  \end{center}
\end{figure}

For each $r$, define $\Sigma_r = h^{-1}(s_r)$.  This family of surfaces defines an isotopy of $\Sigma$ corresponding to some element of $Isot(M, \Sigma)$.  We will modify this isotopy so that at every time $r$, the surface is disjoint from one of the spines $K^\pm$ as follows:

The restriction of $f$ to $\Sigma_r$ is a Morse function on the surface, and the level sets at $a_r$ and $b_r$ form loops in this surface.  In the level surface of $f$, these loops are trivial since $\Sigma$ is mostly above or mostly below $\Sigma'$ at these points.  Thus we can compress $\Sigma$ along these collections of loops to a surface that separates $K^-$ from $K^+$.  Because $\Sigma$ has genus exactly one more than $\Sigma'$, at most one of these loops can be an actual compression. For each $r$, this compression is either at $a_r$ or $b_r$.

At each of the finitely many values of $r$ where the values $a_r$ are discontinuous, let $a_r$ and $a'_r$ be the left and right limits. If the level set at level $b_r$ contains an essential loop then the level sets at both levels $a_r$ and $a'_r$ must be trivial, so neither defines a compression. If either of the level sets at $a_r$ or $a'_r$ does contain a compression then $b_r$ does not. The same argument holds for a jump in $b_r$.

Thus we can cut the interval $[0,1]$ at points $0 = r_0 < r_1 < \dots < r_k = 1$ so that for $r \in [r_i, r_{i+1}]$, the level sets at level $b_r$ are trivial when $i$ is even and the loops at level $a_r$ are trivial when $i$ is odd.  Moreover, we can assume that for each even $i$, there is an $r \in [r_i, r_{i+1}]$ such that the level loops of $a_r$ are non-trivial and vice-versa for odd $i$.  In other words, we want to cut the interval into as few subintervals as possible.

For $r \in [r_0, r_1]$, the level loops in $\Sigma_r$ at level $b_r$ are trivial in $\Sigma_r$.  Each of these loops bounds a disk in each of the surfaces $\Sigma_r$, $\Sigma'_{b_r}$ and there is, up to isotopy, a unique way to project the disk in $\Sigma_r$ onto the disk $\Sigma'_{a_r}$, as in Figure~\ref{fig:makedsjt}.  Let $S_r$ be the result of projecting all the disks in this way and assume we have chosen the projections to vary continuously with $r$ along the interval.
\begin{figure}[htb]
  \begin{center}
  \includegraphics[width=3.5in]{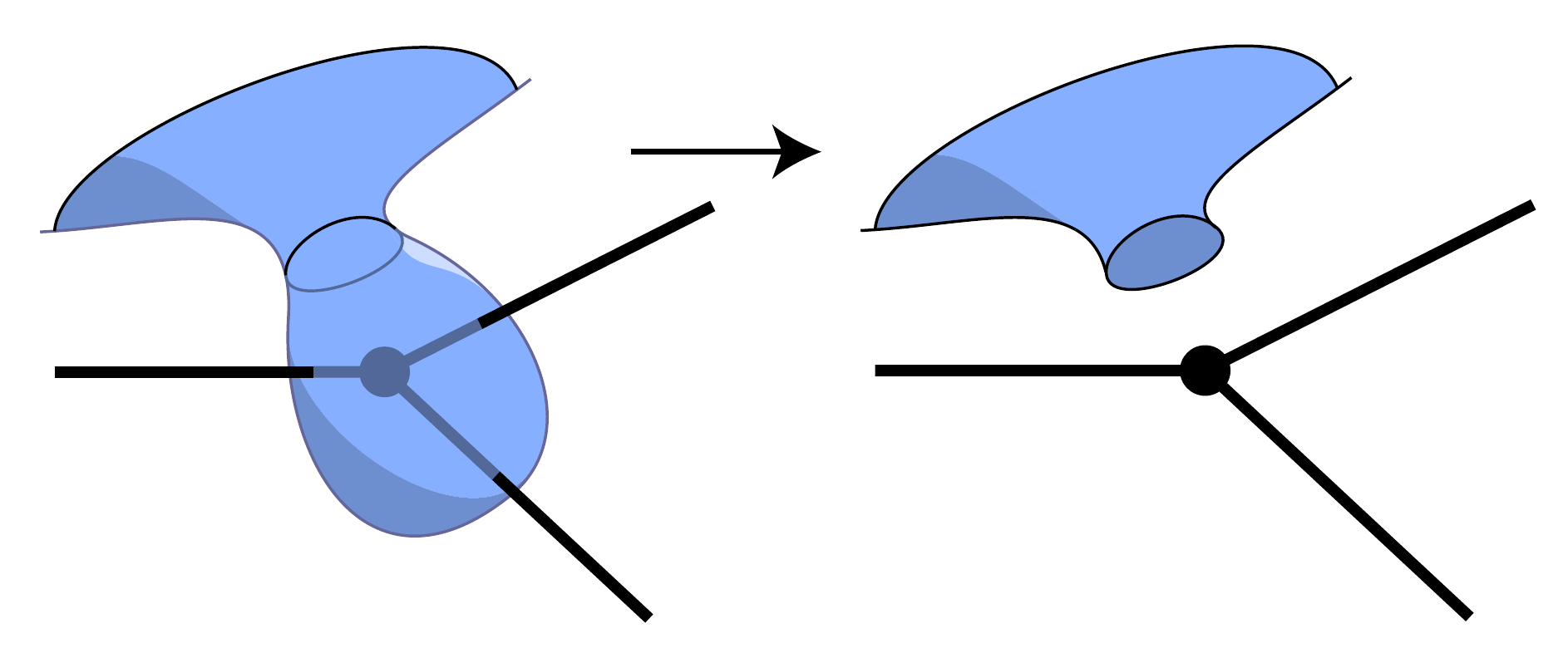}
  \caption{Removing trivial intersections between $\Sigma$ and a spine.}
  \label{fig:makedsjt}
  \end{center}
\end{figure}

After the projection, the surface $S_r$ is entirely below level $b_r$ and is thus disjoint from $K^+$.  For the intervals in which the level loops at level $a_r$ are trivial, $S_r$ will be disjoint from $K^-$.  Repeat this construction for each $i$.  The resulting isotopy defines the same element of $Mod(M, \Sigma)$ as the original isotopy, so all that remains is to show that $\Sigma_{r_i}$ is a Heegaard surface for the complement of both spines for each $i$.

We will start with $r_1$ and then repeat the argument for each $i$.  By assumption, there is an $r \in [r_0, r_1]$ such that the level set of $\Sigma_r$ at level $b_r$ is non-trivial.  If we compress along all these loops, the resulting surface will contain a component isotopic to the boundary of a regular neighborhood of $K^-$, which is on the negative side of $\Sigma_r$.  Thus at least one of the compressions must have been on the negative side of $\Sigma_r$.

Similarly, if we choose $r' \in [r_1, r_2]$, we can find a compressing disk on the positive side of $\Sigma_{r'}$.  If we choose $r$ to be the last such value and $r'$ to be the first such value, then these compressing disks will determine compressing disks on both sides of $\Sigma_{r_1}$.  Since $\Sigma_{r_1}$ is bicompressible, it is reducible by Lemma~\ref{bicomplem}, as in Figure~\ref{fig:knottedhandle}.
\begin{figure}[htb]
  \begin{center}
  \includegraphics[width=3.5in]{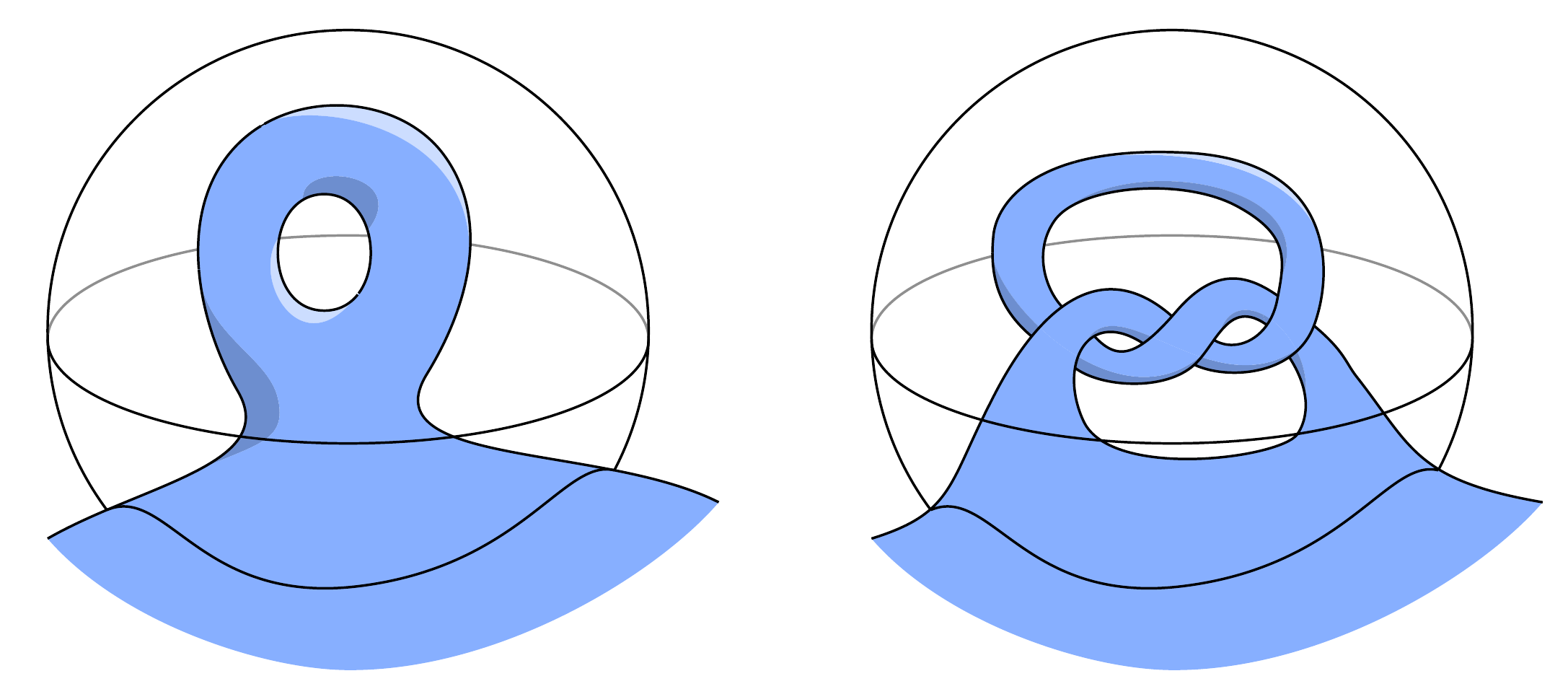}
  \caption{Two possible reducible surfaces in $M \setminus K^+$.}
  \label{fig:knottedhandle}
  \end{center}
\end{figure}

For each $r \in [r_0, r_1]$, the surface $\Sigma_r$ is disjoint from $K^-$.  The surfaces thus determine an isotopy of the surface inside the handlebody $M \setminus K^+$, so each $\Sigma_r$ is a Heegaard surface for this handlebody.  Any reducing sphere for a Heegaard surface in an irreducible three-manifold determines a Heegaard surface for the ball bounded by the reducing sphere. In this case, we get a genus-one Heegaard surface, which is a standard unknotted torus~\cite{Waldhausen}. Thus the reducing sphere for $\Sigma_{r_1}$, which is disjoint from both spines, bounds an unknotted handle, as on the left in Figure~\ref{fig:knottedhandle}, and $\Sigma_{r_1}$ is a Heegaard surface for the complement of the two spines.  By repeating this argument for each successive $i$, we complete the proof.
\end{proof}

\begin{proof}[Proof of Theorem~\ref{mainthm}] 
Let $\gamma \in Isot(M, \Sigma)$ be an element of the isotopy subgroup of $\Sigma$, which is the result of stabilizing a Heegaard surface $\Sigma'$ exactly once.  Assume the distance $d(\Sigma')$ is strictly greater than $2g+2$.  Then by Lemma~\ref{neverbothlem}, we can represent $\gamma$ by an isotopy satisfying the conditions of Lemma~\ref{ifneverbothlem}.  Then by Lemma~\ref{ifneverbothlem}, $\gamma$ is in the subgroup generated by $St^-(\Sigma) \cup St^+(\Sigma) \subset Isot(M, \Sigma)$.  

Since $\gamma$ was an arbitrary element, $St^-(\Sigma) \cup St^+(\Sigma)$ must generate the entire group $Isot(M, \Sigma)$.  Because $d(\Sigma) > 2g+2 > 3$, $M$ is hyperbolic by Hempel's Theorem~\cite{hempel} (and geometrization). Thus $Mod(M)$ is finite, so $Isot(M, \Sigma)$ is a finite index subgroup of $Mod(M, \Sigma)$.  Since a finite index subgroup of $Mod(M, \Sigma)$ is finitely generated, the entire group is finitely generated.
\end{proof}

\bibliographystyle{amsplain}
\bibliography{stabmcg}

\end{document}